\documentclass{proc-l-hijacked}
\usepackage{amssymb, amsmath, amsthm,hyperref}

\newtheorem{theorem}{Theorem}[section]
\newtheorem{lemma}[theorem]{Lemma}
\newtheorem{corollary}[theorem]{Corollary}

\theoremstyle{definition}

\newtheorem{example}[theorem]{Example}

\theoremstyle{remark}

\numberwithin{equation}{section}

\begin{document}
\title[Commutators, Commutativity and Dimension in the Socle]{ Commutators, Commutativity and Dimension in the Socle of a Banach Algebra: A generalized Wedderburn-Artin and Shoda's Theorem  }
\author{F. Schulz \and R. Brits}
\address{Department of Mathematics, University of Johannesburg, South Africa}
\email{francoiss@uj.ac.za, rbrits@uj.ac.za}

\begin{abstract}
As a follow-up to work done in \cite{tracesocleident}, some new insights to the structure of the socle of a semisimple Banach algebra is obtained. In particular, it is shown that the socle is isomorphic as an algebra to the direct sum of tensor products of corresponding left and right minimal ideals. Remarkably,  the finite-dimensional case here reduces to the classical Wedderburn-Artin Theorem, and this approach does not use any continuous irreducible representations of the algebra in question. Furthermore, the structure of the socles for which the classical Shoda's Theorem for matrices can be extended, is characterized exactly as those socles which are minimal two-sided ideals. It is then shown that the set of commutators in the socle (i.e. $\left\{xy-yx : x, y \in \mathrm{Soc}\:A \right\}$) is a vector subspace. Finally, we characterize those socles which belong to the center of a Banach algebra and obtain results which suggests that the dimension of certain subalgebras of the socle in fact provides a measure, to some extent, of commutativity.  
\end{abstract}

\maketitle

\section{Introduction}
The notions of rank, trace and determinant are well-established for operator theory. Rather recently, in their paper entitled \textit{Trace and determinant in Banach algebras} \cite{aupetitmoutontrace}, Aupetit and Mouton managed to show that these notions can be developed, without the use of operators, in a purely spectral and analytic manner. This paper is fundamental to our discussion here, for this alternative point of view not only permits the possibility to consider rank and trace related problems in a more general setting, but also allows for new insights to the structure of the socle of a semisimple Banach algebra. As in \cite{tracesocleident} we briefly summarize some of the theory in \cite{aupetitmoutontrace} before we proceed.\\

By $A$ we denote a complex Banach algebra with identity element $\mathbf 1$ and invertible group $G(A)$. Moreover, it will be assumed throughout that $A$ is semisimple (i.e. the Jacobson radical of $A$, denoted $\mathrm{Rad}\:A$, only contains $0$). By $Z(A)$ we denote the center of $A$, that is, the set of all $x \in A$ such that $xy=yx$ for all $y \in A$. For $x\in A$ we denote by $\sigma_A(x) =\left\{\lambda\in\mathbb C:\lambda\mathbf{1}-x\notin G(A)\right\}$, $\rho_{A} (x) = \sup\left\{\left|\lambda\right| : \lambda \in \sigma_{A} (x)\right\}$ and $\sigma_{A} '(x)=\sigma_A(x) -\{0\}$ the spectrum, spectral radius and nonzero spectrum of $x$, respectively. If the underlying algebra is clear from the context, then we shall agree to omit the subscript $A$ in the notation $\sigma_A(x)$, $\rho_{A} (x)$ and $\sigma_{A} '(x)$. This convention will also be followed in the forthcoming definitions of rank, trace, and so forth. We shall also agree to reserve the notation $\cong$ exclusively for algebra isomorphisms. Any other type of isomorphism will explicitly be referred to.\\ 

For each nonnegative integer $m$, let
$$\mathcal{F}_{m} = \left\{a \in A: \#\sigma'(xa) \leq m \;\,\mathrm{for\;all}\;\,x \in A \right\},$$
where the symbol $\#K$ denotes the number of distinct elements in a set $K\subseteq \mathbb C$. Following Aupetit and Mouton in \cite{aupetitmoutontrace}, we define the \textit{rank} of an element $a$ of $A$ as the smallest integer $m$ such that $a \in \mathcal{F}_{m}$, if it exists; otherwise the rank is infinite. In other words,
$$\mathrm{rank}\,(a) = \sup_{x \in A} \#\sigma'(xa).$$
If $a \in A$ is a finite-rank element, then
$$E(a) = \left\{x \in A : \#\sigma'(xa) = \mathrm{rank}\,(a) \right\}$$
is a dense open subset of $A$ \cite[Theorem 2.2]{aupetitmoutontrace}. A finite-rank element $a$ of $A$ is said to be a \textit{maximal finite-rank element} if $\mathrm{rank}\,(a) = \#\sigma'(a)$. With respect to $\mathrm{rank}$ it is further useful to know that $\sigma ' (xa)=\sigma '(ax)$ for all $x, a \in A$ (Jacobson's Lemma, \cite[Lemma 3.1.2.]{aupetit1991primer}). It can be shown \cite[Corollary 2.9]{aupetitmoutontrace} that the socle, written $\mathrm{Soc}\:A$, of a semisimple Banach algebra $A$ coincides with the collection $\bigcup_{m = 0}^{\infty} \mathcal{F}_{m}$ of finite rank elements. We mention a few elementary properties of the rank of an element \cite[p. 117]{aupetitmoutontrace}. Firstly, $\#\sigma'(a) \leq \mathrm{rank}\,(a)$ for all $a \in A$. Furthermore, $\mathrm{rank}\,(xa) \leq \mathrm{rank}\,(a)$ and $\mathrm{rank}\,(ax)\leq \mathrm{rank}\,(a)$ for all $x, a \in A$, with equality if $x \in G(A)$. Moreover, the rank is lower semicontinuous on $\mathrm{Soc}\:A$. It is also subadditive, i.e. $\mathrm{rank}\,(a+b) \leq \mathrm{rank}\,(a) + \mathrm{rank}\,(b)$ for all $a, b \in A$ \cite[Theorem 2.14]{aupetitmoutontrace}. Finally, if $p$ is a projection of $A$, then $p$ has rank one if and only if $p$ is a minimal projection, that is $pAp = \mathbb{C}p$. It is also worth mentioning here that a projection $p$ is minimal if and only if $Ap$ is a nontrivial left ideal which does not contain any left ideals other than $\left\{0 \right\}$ and itself, that is, if and only if $Ap$ is a nontrivial minimal left ideal \cite[Lemma 30.2]{bonsall1973complete}. A similar result holds true for the right ideal $pA$. We will also define a minimal two-sided ideal in this manner, that is, as a two-sided ideal which does not contain any two-sided ideals other than $\left\{0 \right\}$ and itself.\\

The following result is fundamental to the theory developed in \cite{aupetitmoutontrace} and is mentioned here for convenient referencing later on:\\

\textbf{Diagonalization Theorem} \cite[Theorem 2.8]{aupetitmoutontrace}: Let $a \in A$ be a nonzero maximal finite-rank element and denote by $\lambda_{1}, \ldots, \lambda_{n}$ its nonzero distinct spectral values. Then there exists $n$ orthogonal minimal projections $p_{1}, \ldots, p_{n} \in Aa \cap aA$ such that 
$$a = \lambda_{1}p_{1} + \cdots + \lambda_{n}p_{n}.$$

In particular, the Diagonalization Theorem easily implies the well-known result that every element of the socle is \emph{Von Neumann regular}, that is, for each $a \in \mathrm{Soc}\:A$, there exists an $x \in \mathrm{Soc}\:A \subseteq A$ such that $a = axa$ \cite[Corollary 2.10]{aupetitmoutontrace}. Another useful result is \cite[Theorem 2.16]{aupetitmoutontrace}, which states that for any set of nonzero orthogonal finite-rank projections $\left\{p_{1}, \ldots, p_{n}\right\}$ we have
$$\mathrm{rank}\,\left(\alpha_{1}p_{1} + \cdots +\alpha_{n}p_{n}\right) = \mathrm{rank}\,\left(p_{1}\right) + \cdots +\mathrm{rank}\,\left(p_{n}\right)$$
for all nonzero complex numbers $\alpha_{1}, \ldots, \alpha_{n}$.\\

If $a \in \mathrm{Soc}\:A$ we define the \textit{trace} of $a$ as in \cite{aupetitmoutontrace} by
$$\mathrm{Tr}\,(a) = \sum_{\lambda \in \sigma (a)} \lambda m\left(\lambda, a\right),$$
where $m(\lambda,a)$ is the \emph{multiplicity of $a$ at $\lambda$}. A brief description of the notion of multiplicity in the abstract case goes as follows (for particular details one should consult \cite{aupetitmoutontrace}): Let $a \in \mathrm{Soc}\:A$, $\lambda\in\sigma(a)$ and let $B(\lambda,r)$ be an open disk centered at $\lambda$ such that $B(\lambda,r)$ contains no other points of $\sigma(a)$. It can be shown \cite[Theorem 2.4]{aupetitmoutontrace} that there exists an open ball, say $U\subseteq A$, centered at $\mathbf{1}$ such that $\#\left[\sigma(xa)\cap B(\lambda,r) \right]$ is constant as $x$ runs through $E(a)\cap U$. This constant integer is the multiplicity of $a$ at $\lambda$. It can also be shown that $m\left(\lambda,a\right) \geq 1$ and
\begin{equation}
\sum_{\alpha \in \sigma (a)} m(\alpha, a) = \left\{\begin{array}{cl} 1+\mathrm{rank}\,(a) & \mathrm{if}\;\,0 \in \sigma (a) \\
\mathrm{rank}\,(a) & \mathrm{if}\;\,0 \notin \sigma (a).
\end{array}\right.
\label{beq0}
\end{equation} 
Furthermore, we note that the trace has the following useful properties:
\begin{itemize}
\item[(i)]
$\mathrm{Tr}$ is a linear functional on $\mathrm{Soc}\:A$ (\cite[Theorem 3.3]{aupetitmoutontrace} and \cite[Lemma 2.1]{tracesocleident}).
\item[(ii)]
$\mathrm{Tr}\,(ab) = \mathrm{Tr}\,(ba)$ for each $a \in \mathrm{Soc}\:A$ and $b \in A$ \cite[Corollary 2.5]{tracesocleident}.
\item[(iii)]
For any $a \in \mathrm{Soc}\:A$, if $\mathrm{Tr}\,(xa) = 0$ for each $x \in \mathrm{Soc}\:A$, then $a = 0$ \cite[Corollary 3.6]{aupetitmoutontrace}.
\end{itemize}

Let $\lambda \in \sigma (a)$ and suppose that $B(\lambda,2r)$ separates $\lambda$ from the remaining spectrum of $a$. Let $f_{\lambda}$ be the holomorphic function which takes the value $1$ on $B(\lambda,r)$ and the value $0$ on $\mathbb{C} - \overline{B}(\lambda,r)$. If we now let $\Gamma_{0}$ be a smooth contour which surrounds $\sigma (a)$ and is contained in the domain of $f_{\lambda}$, then
$$p\left(\lambda, a\right) = f_{\lambda} (a) = \frac{1}{2\pi i}\int_{\Gamma_{0}} f_{\lambda} (\alpha) \left(\alpha \mathbf{1}-a\right)^{-1}\,d\alpha$$
is referred to as the \textit{Riesz projection} associated with $a$ and $\lambda$. By the Holomorphic Functional Calculus, Riesz projections associated with $a$ and distinct spectral values are orthogonal and for $\lambda \neq 0$
\begin{equation}
p\left(\lambda, a\right) = \frac{a}{2\pi i} \int_{\Gamma_{0}} \frac{f_{\lambda}\left(\alpha\right)}{\alpha} \left(\alpha\mathbf{1} - a\right)^{-1}\,d\alpha \in Aa \cap aA. \label{beq1}
\end{equation}
The following results will also be useful: Let $a \in A$ have finite rank and let $\lambda_{1}, \ldots, \lambda_{n}$ be nonzero distinct elements of its spectrum. If 
$$p = p\left(\lambda_{1}, a\right) + \cdots + p\left(\lambda_{n}, a\right),$$ then by \cite[Theorem 2.6]{aupetitmoutontrace} we have
$$\mathrm{rank}\,(p) = m\left(\lambda_{1}, a\right) + \cdots + m\left(\lambda_{n}, a\right).$$
Moreover, $\mathrm{rank}\,(p) = m\left(1, p\right)$ \cite[Corollary 2.7]{aupetitmoutontrace}. It is customary to refer to $p$ here as the Riesz projection associated with $a$ and $\lambda_{1}, \ldots , \lambda_{n}$.\\

In the operator case, $A=B(X)$ (bounded linear operators on a Banach space $X$), the ``spectral" rank and trace both coincide with the respective classical operator definitions.

\section{Preliminaries}

If $p$ is a projection of $A$, then $pAp$ is a closed semisimple subalgebra of $A$ with identity $p$ \cite[Exercise 3.6]{aupetit1991primer}. The subalgebra $pAp$ is very useful in the theory of rank, trace and determinant, primarily because of the following reasons:
\begin{equation}
\sigma_{pAp}' \left(pxp\right) = \sigma_{A}' \left(pxp\right)
\label{beq2.1}
\end{equation}
and
\begin{equation}
\mathrm{rank}_{pAp}(pxp) = \mathrm{rank}_{A}(pxp)
\label{beq2.2}
\end{equation}
for each $x \in A$. The proof of (\ref{beq2.1}) is not hard and (\ref{beq2.2}) is a consequence of (\ref{beq2.1}) and Jacobson's Lemma.\\

Let $p$ be a projection of $A$ with $\mathrm{rank}\,(p) \leq 1$. By $J_{p}$ we denote the two-sided ideal generated by $p$, that is, we let
$$J_{p} := \left\{ \sum_{j=1}^{n} x_{j}py_{j} : x_{j}, y_{j} \in A, n \geq 1\;\,\mathrm{an\;integer} \right\}.$$ 

The next two lemmas are well-known results in Ring Theory. However, since these results will play a central role in the development of the subsequent theory, we provide a short proof of the latter and refer the reader to the recent reference \cite{tracesocleident} for a proof of the former.  

\begin{lemma}\label{b2.1}\cite[Lemma 3.5]{tracesocleident}
There exists a collection of two-sided ideals $\left\{J_{p} : p \in \mathcal{P} \right\}$ such that every element of  $\mathrm{Soc}\:A$ can be written as a finite sum of members of the $J_{p}$. Moreover, the two-sided ideals are pairwise orthogonal, that is, if $p, q \in \mathcal{P}$ with $p \neq q$, then
$$J_{p}J_{q} = J_{q}J_{p} = \left\{0 \right\}.$$
\end{lemma}

\begin{lemma}\label{b2.2}
Let $p$ be a projection of $A$ with $\mathrm{rank}\,(p) \leq 1$. Then $J_{p}$ is a minimal two-sided ideal.
\end{lemma}

\begin{proof}
If $p = 0$, then the result is obviously true. So assume that $p \neq 0$. Let $\left\{0 \right\} \neq J \subseteq J_{p}$ be a two-sided ideal. We claim that $J = J_{p}$: It will suffice to show that $p \in J$. By hypothesis, there exists an $a \in J$ such that $a \neq 0$ and $a = \sum_{j = 1}^{n} x_{j}py_{j}$ for some $x_{1}, \ldots, x_{n}, y_{1}, \ldots, y_{n} \in A$. Assume that $pxayp = 0$ for all $x, y \in A$. Then, by Jacobson's Lemma and the semisimplicity of $A$ it follows that $ayp = 0$ for all $y \in A$. However, then $awa = \sum_{j = 1}^{n} awx_{j}py_{j} = 0$ for all $w \in A$. But $a$ is Von Neumann regular, so this implies that $a = 0$ which is absurd. Therefore, $px_{0}ay_{0}p \neq 0$ for some $x_{0}, y_{0} \in A$. Since $p$ has rank one, it is minimal. Hence, $px_{0}ay_{0}p = \lambda p$ for some $\lambda \in \mathbb{C}-\left\{0 \right\}$. Thus, $p \in J$. This proves our claim which gives the result.
\end{proof}



\begin{theorem}\label{b2.3}
Let $p$ be a projection of $A$ with $\mathrm{rank}\,(p) \leq 1$. Then there exists a linear isomorphism from $Ap \otimes pA$ onto $J_{p}$. 
\end{theorem}

\begin{proof}
If $p = 0$ then the result is trivially true. So assume that $p$ is a rank one projection. Let $\tau: Ap \oplus pA \rightarrow Ap \otimes pA$ be the \emph{tensor map}, that is, let
$$\tau (x, y) = x \otimes y \,\;\:\left(x \in Ap, y \in pA\right).$$
Next we consider the mapping $\psi: Ap \oplus pA \rightarrow J_{p}$ given by
$$\psi(x, y) = xpy \;\,\:\left(x \in Ap, y \in pA\right).$$
It is routine to show that $\psi$ is a bilinear mapping. Thus, by \cite[Theorem 42.6]{bonsall1973complete} there exists a unique linear mapping $\phi: Ap \otimes pA \rightarrow J_{p}$ such that $\psi = \phi \circ  \tau$. It will therefore suffice to show that $\phi$ is bijective. Let $x \in J_{p}$ be arbitrary. Then $x = \sum_{j=1}^{n} u_{j}pv_{j}$ for some $u_{1}, \ldots, u_{n}, v_{1}, \ldots, v_{n} \in A$. By the linearity of $\phi$ and the fact that $p = p^{2}$, it readily follows that
$$\phi \left( \sum_{j=1}^{n} u_{j}p \otimes pv_{j} \right) = x.$$
This proves that $\phi$ is surjective. To see that $\phi$ is injective, suppose that
\begin{equation}
\phi \left( \sum_{j=1}^{n} x_{j} \otimes y_{j} \right) = 0,
\label{beq2}
\end{equation}
where $x_{1}, \ldots, x_{n} \in Ap$ and $y_{1}, \ldots, y_{n} \in pA$. By \cite[Lemma 42.3]{bonsall1973complete} we may assume without loss of generality that $\left\{x_{1}, \ldots, x_{n} \right\}$ and $\left\{y_{1}, \ldots, y_{n} \right\}$ are linearly independent subsets of $Ap$ and $pA$, respectively. From (\ref{beq2}) it follows that
\begin{equation}
x_{1}py_{1} + \cdots x_{n}py_{n} = 0.
\label{beq3}
\end{equation}
Using (\ref{beq3}), the minimality of $p$, the fact that $x_{i}p = x_{i}$ for each $i \in \left\{1, \ldots, n \right\}$, and the linear independence of $\left\{x_{1}, \ldots, x_{n} \right\}$, we may conclude that for any $j \in \left\{1, \ldots, n \right\}$, we have 
\begin{equation}
py_{j}wp = 0\,\; \mathrm{for\;all}\;\, w \in A.
\label{beq4}
\end{equation}
Fix any $j \in \left\{1, \ldots, n \right\}$ and let $y \in A$ be arbitrary. Since $y_{j} = py_{j}$, it follows from (\ref{beq4}) and Jacobson's Lemma that $\sigma \left(yy_{j}\right) = \left\{0 \right\}$. Hence, by the semisimplicity of $A$ it follows that $y_{j} = 0$ for each $j \in \left\{ 1, \ldots, n\right\}$. From this it now follows that $x_{j} \otimes y_{j} = 0$ for all $j \in \left\{ 1, \ldots, n\right\}$. Thus, $\sum_{j=1}^{n} x_{j} \otimes y_{j} = 0$, and so, $\phi$ is injective. This completes the proof.
\end{proof}

By definition $Ap \otimes pA$ is a vector space. However, Theorem \ref{b2.3} readily gives the following result:

\begin{corollary}\label{b2.4}
Let $\phi: Ap \otimes pA \rightarrow J_{p}$ be the linear isomorphism obtained in Theorem \textnormal{\ref{b2.3}}. Define multiplication in $Ap \otimes pA$ by letting
\begin{equation}
uv = \phi^{-1}\left(\phi(u)\phi(v)\right)\;\,\:\left(u, v \in Ap \otimes pA\right).
\label{beq5}
\end{equation}
Then $Ap \otimes pA$ equipped with the above multiplication scheme is an algebra.
\end{corollary}

\begin{proof}
There is nothing ambiguous about the right-side of (\ref{beq5}), so the multiplication is well-defined. Moreover, the algebra properties are inherited from $J_{p}$ which gives the result.
\end{proof}

It is interesting to note that the multiplication scheme above for elementary tensors in $Ap \otimes pA$ is actually given by
$$\left(x_{1} \otimes y_{1}\right)\left(x_{2} \otimes y_{2}\right) = \mathrm{Tr}\,\left(y_{1}x_{2}\right) \left(x_{1} \otimes y_{2}\right) \;\,\:\left(x_{1}, x_{2} \in Ap, y_{1}, y_{2} \in pA\right).$$
Moreover, since $\phi$ obtained in Theorem \ref{b2.3} is a linear isomorphism and
$$\phi^{-1}\left(\phi(u)\right)\phi^{-1}\left(\phi(v)\right) = uv = \phi^{-1}\left(\phi(u)\phi(v)\right)$$
for all $u, v \in Ap \otimes pA$, we readily obtain the following result:

\begin{corollary}\label{b2.5}
Let $p$ be a projection of $A$ with $\mathrm{rank}\,(p) \leq 1$. Then $Ap \otimes pA \cong J_{p}$. 
\end{corollary}

\section{The Wedderburn-Artin Theorem}

Let $\mathcal{P}$ be the class of projections generating the $J_{p}$ from Lemma \ref{b2.1}. By the direct sum $\oplus_{p \in \mathcal{P}} Ap \otimes pA$ we denote the subset of the Cartesian product $\times_{p \in \mathcal{P}}Ap \otimes pA$ consisting of all cross sections which are zero except at a finite number of elements of $\mathcal{P}$, equipped with pointwise scalar multiplication, addition and multiplication. From Lemma \ref{b2.1} and Corollary \ref{b2.5} we obtain the following theorem concerning the structure of the socle:

\begin{theorem}\label{b3.1}
$\mathrm{Soc}\:A \cong \oplus_{p \in \mathcal{P}} Ap \otimes pA$.
\end{theorem}

The result above can be viewed as a generalization of the celebrated \emph{Wedderburn-Artin Theorem} \cite[Theorem 2.1.2]{aupetit1991primer}. Of course, any claim of generality should inherently contain the classical result in some or other form. We now proceed to show that Theorem \ref{b3.1} is indeed a generalized version of the Wedderburn-Artin Theorem. Noteworthy is that this approach not only yields further interesting consequences, but it also avoids the use of representation theory altogether, that is, we manage to prove the Wedderburn-Artin Theorem without the use of continuous irreducible representations of $A$. Firstly, however, a little preparation is needed:

\begin{lemma}\label{b3.2}
Let $p$ be a rank one projection of $A$ and let $S$ be a linearly independent subset of $Ap$ such that $p \in S$. Then
$$S' = \left\{ p\right\} \cup \left\{\left(\mathbf{1}-p\right)xp: x \in S-\left\{p\right\} \right\}$$
is a linearly independent subset of $Ap$ and $\mathrm{span}\;S = \mathrm{span}\;S'$. A similar result is true for $pA$.
\end{lemma}

\begin{proof}
Let $x_{2}, \ldots, x_{n} \in S-\left\{p\right\}$ be distinct. For the first part it will suffice to show that the set $\left\{p, \left(\mathbf{1}-p\right)x_{2}p, \ldots, \left(\mathbf{1}-p\right)x_{n}p  \right\}$ is linearly independent. To this end, suppose that
\begin{equation}
\lambda_{1}p + \lambda_{2}\left(\mathbf{1}-p\right)x_{2}p + \cdots + \lambda_{n}\left(\mathbf{1}-p\right)x_{n}p = 0
\label{beq6}
\end{equation}
for some $\lambda_{1}, \ldots, \lambda_{n} \in \mathbb{C}$. Since $p$ has rank one, it is minimal. Hence, for each $i \in \left\{2, \ldots, n \right\}$ we have that $px_{i}p = \alpha_{i}p$, where $\alpha_{i} \in \mathbb{C}$. Moreover, since $x_{i} \in Ap$, it follows that $x_{i}p = x_{i}$ for each $i \in \left\{2, \ldots, n \right\}$. Hence, (\ref{beq6}) becomes
$$\left(\lambda_{1} - \lambda_{2}\alpha_{2} - \cdots -\lambda_{n}\alpha_{n}\right)p + \lambda_{2}x_{2} + \cdots + \lambda_{n}x_{n} = 0.$$
Thus, since $\left\{p, x_{2}, \ldots, x_{n} \right\}$ is a linearly independent set, it readily follows that $\lambda_{i} = 0$ for each $i \in \left\{1, \ldots, n \right\}$. The second part follows from the minimality of $p$ and the observation that
$$\lambda_{1}p + \lambda_{2}x_{2} + \cdots + \lambda_{n}x_{n} = \left(\lambda_{1} + \sum_{j=2}^{n}\lambda_{j}\alpha_{j}\right)p + \lambda_{2}\left(\mathbf{1}-p\right)x_{2}p + \cdots + \lambda_{n}\left(\mathbf{1}-p\right)x_{n}p$$
and
$$\lambda_{1}p + \lambda_{2}\left(\mathbf{1}-p\right)x_{2}p + \cdots + \lambda_{n}\left(\mathbf{1}-p\right)x_{n}p = \left(\lambda_{1} - \sum_{j=2}^{n}\lambda_{j}\alpha_{j}\right)p + \lambda_{2}x_{2} + \cdots + \lambda_{n}x_{n}$$
for any $\lambda_{1}, \ldots, \lambda_{n} \in \mathbb{C}$, where $px_{j}p = \alpha_{j}p$ for each $j \in \left\{2, \ldots, n \right\}$. This gives the result.
\end{proof}

From \cite[pp. 155--157]{bonsall1973complete} we recall that associated to every rank one element, $a$, there exists a characteristic functional $\tau_{a} \in A'$ (the dual space of $A$) such that
$$axa = \tau_{a}\left(x\right) a \,\;\mathrm{for\;all}\;\, x \in A.$$
Observe that $\alpha\tau_{a} = \tau_{\alpha a}$ for all $\alpha \in \mathbb{C}-\left\{0 \right\}$. Moreover, by the density of $E(a)$ and the Diagonalization Theorem, there exist a minimal projection $p$ in $A$ and a $u \in G(A)$ such that $a=pu$. Thus, by Jacobson's Lemma, and the definitions of $\mathrm{Tr}$ and $\tau_{a}$, it readily follows that 
$$\mathrm{Tr}\,(ax) = \tau_{a} (x)\,\;\mathrm{for\;all}\;\, x \in A.$$

\begin{theorem}\label{b3.3}
Let $\left\{b, a_{1}, \ldots, a_{n}\right\}$ be a linearly independent set of rank one elements in $A$. Then there exists a $y \in A$ such that $\sigma (by) \neq \left\{0 \right\}$ and $\sigma \left(a_{i}y\right) = \left\{0 \right\}$ for each $i \in \left\{1, \ldots, n \right\}$.
\end{theorem}

\begin{proof}
For the sake of a contradiction, suppose that
\begin{equation}
\sigma \left(a_{i}x\right) = \left\{0 \right\} \;\,\mathrm{for\;all}\;\,i \in \left\{1, \ldots, n \right\} \Rightarrow \sigma (bx) = \left\{0 \right\}.
\label{beq10}
\end{equation}
Now, (\ref{beq10}) is clearly equivalent to
$$x \in \bigcap_{i =1}^{n} \mathrm{Ker}\:\tau_{a_{i}} \Rightarrow x \in \mathrm{Ker}\:\tau_{b},$$
which means that $\tau_{b}$ vanishes on $\bigcap_{i =1}^{n} \mathrm{Ker}\:\tau_{a_{i}}$ i.e. $\mathrm{Ker}\:\tau_{b} \supseteq \bigcap_{i =1}^{n} \mathrm{Ker}\:\tau_{a_{i}}$. Consequently, from linear algebra (see \cite[p. 10]{diestel1984sequences}) it follows that $\tau_{b} = \alpha_{1}\tau_{a_{1}} + \cdots +\alpha_{n}\tau_{a_{n}}$, where $\alpha_{1}, \ldots, \alpha_{n} \in \mathbb{C}$. Next, we let $x \in A$ be arbitrary, and consider
\begin{eqnarray*}
\mathrm{Tr}\,(bx) = \tau_{b} (x) & = & \left(\alpha_{1}\tau_{a_{1}} + \cdots +\alpha_{n}\tau_{a_{n}}\right)(x) \\
& = & \alpha_{1}\tau_{a_{1}}(x) + \cdots +\alpha_{n}\tau_{a_{n}}(x) \\
& = & \tau_{\alpha_{1}a_{1}}(x) + \cdots +\tau_{\alpha_{n}a_{n}}(x) \\
& = & \mathrm{Tr}\,\left(\alpha_{1}a_{1}x\right) + \cdots + \mathrm{Tr}\,\left(\alpha_{n}a_{n}x\right) \\
& = & \mathrm{Tr}\,\left(\alpha_{1}a_{1}x + \cdots +\alpha_{n}a_{n}x\right) \\
& = & \mathrm{Tr}\,\left(\left(\alpha_{1}a_{1} + \cdots +\alpha_{n}a_{n}\right)x\right).
\end{eqnarray*}
From this it now follows that
$$\mathrm{Tr}\,\left(\left(b-\sum_{i =1}^{n}\alpha_{i}a\right)x\right) = 0.$$
However, since $x \in A$ was arbitrary, it follows from property (iii) of trace (see Section 1) that $b = \sum_{i =1}^{n}\alpha_{i}a$. But this is absurd since $\left\{b, a_{1}, \ldots, a_{n}\right\}$ is linearly independent. We conclude that there is at least one $y \in A$ for which $\sigma (by) \neq \left\{0 \right\}$ and $\sigma \left(a_{i}y\right) = \left\{0 \right\}$ for each $i \in \left\{1, \ldots, n \right\}$. 
\end{proof}

\begin{lemma}\label{b3.4}
Let $p$ be a rank one projection of $A$ and let 
$$\left\{p, \left(\mathbf{1}-p\right)x_{2}p, \ldots,\left(\mathbf{1}-p\right)x_{n}p  \right\}$$
be a linearly independent subset of $Ap$. For each $j \in \left\{2, \ldots, n \right\}$, there exists a $y_{j} \in A$ such that $1 \in \sigma\left(\left(\mathbf{1}-p\right)x_{j}py_{j} \right)$ and $\sigma\left(\left(\mathbf{1}-p\right)x_{i}py_{j} \right) = \left\{0 \right\}$ for $i \neq j$. A similar result holds true if we start with a linearly independent subset of $pA$. 
\end{lemma}

\begin{proof}
This is an immediate consequence of Theorem \ref{b3.3} and the Spectral Mapping Theorem \cite[Theorem 3.3.3]{aupetit1991primer}.
\end{proof}

\begin{lemma}\label{b3.5}
Let $\left\{y_{2}, \ldots, y_{n} \right\}$ be the set obtained in the conclusion of Lemma \textnormal{\ref{b3.4}}. Then $\left\{p, py_{2}\left(\mathbf{1}-p\right), \ldots,py_{n}\left(\mathbf{1}-p\right)  \right\}$ is a linearly independent subset of $pA$. A similar result holds true if we start with a linearly independent subset of $pA$ in Lemma \textnormal{\ref{b3.4}}.
\end{lemma}

\begin{proof}
By Lemma \ref{b3.2} it suffices to show that the set $\left\{p, py_{2}, \ldots, py_{n} \right\}$ is linearly independent in $pA$: Firstly, note that $py_{j} \neq 0$ for each $j \in \left\{2, \ldots, n\right\}$, for otherwise we obtain a contradiction with the fact that $1 \in \sigma\left(\left(\mathbf{1}-p\right)x_{j}py_{j} \right)$ for each $j \in \left\{2, \ldots, n\right\}$. Secondly, by Jacobson's Lemma and the minimality of $p$ it readily follows that $py_{j}\left(\mathbf{1}-p\right)x_{j}p = p$ for each $j \in \left\{2, \ldots, n\right\}$, and, moreover, that $py_{i}\left(\mathbf{1}-p\right)x_{j}p = 0$ for $i \neq j$. Now, suppose that
\begin{equation}
\lambda_{1}p + \lambda_{2}py_{2} + \cdots + \lambda_{n}py_{n} = 0
\label{beq7}
\end{equation}
for some $\lambda_{1}, \ldots, \lambda_{n} \in \mathbb{C}$. Let $j \in \left\{2, \ldots, n\right\}$ be arbitrary but fixed. Multiplying both sides of (\ref{beq7}) on the right by $\left(\mathbf{1}-p\right)x_{j}p$ yields $\lambda_{j}p = 0$. Hence, $\lambda_{j} = 0$ for each $j \in \left\{2, \ldots, n\right\}$. Thus, (\ref{beq7}) becomes $\lambda_{1}p = 0$ from which we get $\lambda_{1} = 0$. Therefore, $\left\{p, py_{2}, \ldots, py_{n} \right\}$ is a linearly independent set as desired.
\end{proof}

\begin{lemma}\label{b3.6}
Let $p$ be a rank one projection of $A$. Then $\mathrm{dim}\:Ap = \mathrm{dim}\:pA$.
\end{lemma}

\begin{proof}
If $\mathrm{dim}\:pA > \mathrm{dim}\:Ap$ or $\mathrm{dim}\:Ap > \mathrm{dim}\:pA$, then in either case we may use Lemma \ref{b3.2} and Lemma \ref{b3.5} to construct a linearly independent set of sufficiently large cardinality in $Ap$ or $pA$, respectively, to produce a contradiction. Hence, it must be the case that $\mathrm{dim}\:pA \leq \mathrm{dim}\:Ap$ and $\mathrm{dim}\:Ap \leq \mathrm{dim}\:pA$, establishing the result. 
\end{proof}

\begin{theorem}\label{b3.7}
Let $p$ be a rank one projection of $A$ and suppose that either $\mathrm{dim}\:Ap$ or $\mathrm{dim}\:pA$ is finite. Then there exist a basis $\left\{p, u_{2}, \ldots, u_{n} \right\}$ of $Ap$ and a basis $\left\{p, v_{2}, \ldots, v_{n} \right\}$ of $pA$ such that the following properties hold:
\begin{itemize}
\item[\textnormal{(i)}]
$pu_{i}=v_{i}p=0$ and $u_{i}^{2} = v_{i}^{2} = 0$ for each $i \in \left\{2, \ldots, n \right\}$.
\item[\textnormal{(ii)}]
$u_{i}p = u_{i}$ and $pv_{i} = v_{i}$ for each $i \in \left\{2, \ldots, n \right\}$.
\item[\textnormal{(iii)}]
$v_{i}u_{i}=p$ for each $i \in \left\{2, \ldots, n \right\}$ and $v_{i}u_{j} = 0$ for $i \neq j$.
\end{itemize}
\end{theorem}

\begin{proof}
From Lemma \ref{b3.6} it follows that $\mathrm{dim}\:Ap = \mathrm{dim}\:pA$. If $\mathrm{dim}\:Ap = 1$, then $\left\{p \right\}$ is a basis for $Ap$ and for $pA$. However, then the result is trivially true. So assume that $\mathrm{dim}\:Ap \geq 2$ and let $\left\{p, x_{2}, \ldots, x_{n} \right\}$ be any basis for $Ap$. From Lemma \ref{b3.2}  it follows that $\left\{p, \left(\mathbf{1}-p\right)x_{2}p, \ldots, \left(\mathbf{1}-p\right)x_{n}p \right\}$ is a basis for $Ap$. Let $\left\{p, py_{2}\left(\mathbf{1}-p\right), \ldots, py_{n}\left(\mathbf{1}-p\right) \right\}$ be the basis for $pA$ constructed in Lemma \ref{b3.4} and Lemma \ref{b3.5} using the aforementioned basis for $Ap$. Recall that Jacobson's Lemma and the minimality of $p$ gives that $py_{i}\left(\mathbf{1}- p\right)x_{i}p = p$ for each $i \in \left\{2, \ldots, n \right\}$ and that $py_{i}\left(\mathbf{1}- p\right)x_{j}p = 0$ for $i \neq j$. Consequently, if we take $u_{i} = \left(\mathbf{1}- p\right)x_{i}p$ and $v_{i} = py_{i}\left(\mathbf{1}- p\right)$ for each $i \in \left\{1, \ldots, n \right\}$, then a moment's investigation shows that $\left\{p, u_{2}, \ldots, u_{n} \right\}$ and $\left\{p, v_{2}, \ldots, v_{n} \right\}$ satisfy properties (i) to (iii).
\end{proof}

\begin{lemma}\label{b3.8}
Let $p$ be a rank one projection of $A$ and suppose that either $\mathrm{dim}\:Ap$ or $\mathrm{dim}\:pA$ is finite. Then $J_{p} \cong M_{n}\left(\mathbb{C}\right)$.
\end{lemma}

\begin{proof}
From Lemma \ref{b3.6} it follows that $\mathrm{dim}\:Ap = \mathrm{dim}\:pA$. If $\mathrm{dim}\:Ap = 1$, then $J_{p} = \mathbb{C}p \cong \mathbb{C}$ and we are done. So assume that $\mathrm{dim}\:Ap \geq 2$. By Theorem \ref{b3.7} we can find a basis $\left\{p, u_{2}, \ldots, u_{n} \right\}$ of $Ap$ and a basis $\left\{p, v_{2}, \ldots, v_{n} \right\}$ of $pA$ satisfying properties (i) to (iii) as listed there. From Corollary \ref{b2.5} it follows that $J_{p} \cong Ap \otimes pA$. Moreover, by \cite[Lemma 42.5]{bonsall1973complete} it follows that
$$\left\{a \otimes b: a \in \left\{p, u_{2}, \ldots, u_{n} \right\}, b \in \left\{p, v_{2}, \ldots, v_{n} \right\} \right\}$$
is a basis for $Ap \otimes pA$. Hence,
$$\left\{p \right\} \cup \left\{u_{i}pv_{j}, u_{i}p, pv_{j}: i, j \in \left\{2, \ldots, n\right\}  \right\}$$
is a basis for $J_{p}$. Let $e_{i, j}$ be the $n \times n$ matrix with $1$ in its $\left(i, j\right)$-entry and $0$ everywhere else. Then, in particular, $\left\{e_{i, j} : i, j \in \left\{1, \ldots, n\right\} \right\}$ is a basis for $M_{n}\left(\mathbb{C}\right)$. We define a linear mapping $\phi : J_{p} \rightarrow M_{n}\left(\mathbb{C}\right)$ in terms of basis elements as follows: $\phi (p) = e_{1, 1}$, $\phi \left(u_{i}p\right) = e_{i, 1}$, $\phi \left(pv_{j}\right) = e_{1, j}$ and $\phi \left(u_{i}pv_{j}\right) = e_{i, j}$ for each $i, j \in \left\{2, \ldots, n\right\}$. It is not hard to show that $\phi$ is bijective. Moreover, by properties (i) to (iii) in Theorem \ref{b3.7} it readily follows that $\phi$ is multiplicative. This gives the result. 
\end{proof}

\begin{theorem}\label{b3.9}
Let $\mathrm{Soc}\:A$ be finite-dimensional. Then 
$$\mathrm{Soc}\:A \cong M_{n_{1}}\left(\mathbb{C}\right) \oplus \cdots \oplus M_{n_{k}}\left(\mathbb{C}\right).$$
\end{theorem}

\begin{proof}
By \cite[Theorem 2.2]{tracesocleident} there exists a $c > 0$ such that
\begin{equation}
\left|\mathrm{Tr}\,(a)\right| \leq c \cdot \rho (a)
\label{beq8}
\end{equation}
for all $a \in \mathrm{Soc}\:A$. Suppose that the collection of two-sided ideals $\left\{J_{p} : p \in \mathcal{P} \right\}$, which exists by Lemma \ref{b2.1}, contains infinitely many elements. However, then $\mathcal{P}$ contains a subset of $n$ distinct pairwise orthogonal rank one projections, say $\left\{q_{1}, \ldots, q_{n} \right\}$, where $n > c$. By the orthogonality of the $q_{i}$, it follows that $q = q_{1} + \cdots +q_{n}$ is also a projection. But then
$$\left|\mathrm{Tr}\,(q)\right| = n > c = c \cdot 1 = c \cdot \rho (q),$$
contradicting (\ref{beq8}). We may therefore conclude that $\mathcal{P}$ is a finite set, say $\mathcal{P} = \left\{p_{1}, \ldots, p_{k} \right\}$. Since the two-sided ideals are pairwise orthogonal, it follows that
$$\mathrm{Soc}\:A \cong J_{p_{1}} \oplus \cdots \oplus J_{p_{k}}.$$
But by Lemma \ref{b3.8} it follows that $J_{p_{i}} \cong M_{n_{i}} \left(\mathbb{C}\right)$ for each $i \in \left\{1, \ldots, k \right\}$. This yields the result. 
\end{proof}

From results in \cite{puhltrace} it is easy to deduce that $\mathrm{Soc}\:A = A$ if and only if $A$ is finite-dimensional. Hence, Theorem \ref{b3.9} readily gives the following:

\begin{corollary}\label{b3.10} \textnormal{(Wedderburn-Artin.)}\\
Let $A$ be finite-dimensional. Then 
$$A \cong M_{n_{1}}\left(\mathbb{C}\right) \oplus \cdots \oplus M_{n_{k}}\left(\mathbb{C}\right).$$
\end{corollary}

To conclude this section, we shall prove two more lemmas which will lead to an illuminative theorem. In particular, this theorem states that any finite collection of socle elements is contained in a subalgebra $B$ of $\mathrm{Soc}\:A$ which has the Wedderburn-Artin structure; that is,
$$B \cong M_{n_{1}}\left(\mathbb{C}\right) \oplus \cdots \oplus M_{n_{k}}\left(\mathbb{C}\right).$$

\begin{lemma}\label{b3.11}
Let $S = \left\{p, \left(\mathbf{1}-p\right)a_{2}p, \ldots,\left(\mathbf{1}-p\right)a_{n}p  \right\}$ and 
$$T = \left\{p, pb_{2}\left(\mathbf{1}-p\right), \ldots,pb_{m}\left(\mathbf{1}-p\right)  \right\}$$ 
be linearly independent subsets of $Ap$ and $pA$, respectively. Then there exist two linearly independent subsets $S' = \left\{p, u_{2}, \ldots, u_{k} \right\}$ and $T' = \left\{p, v_{2}, \ldots, v_{k} \right\}$ of $Ap$ and $pA$, respectively, such that $\mathrm{span}\;S \subseteq \mathrm{span}\;S'$, $\mathrm{span}\;T \subseteq \mathrm{span}\;T'$ and properties \textnormal{(i)} to \textnormal{(iii)} in Theorem \textnormal{\ref{b3.7}} holds for $S'$ and $T'$.
\end{lemma}

\begin{proof}
Apply Lemma \ref{b3.4} and Lemma \ref{b3.5} to $S$ in order to obtain a corresponding linearly independent set $\left\{p, pc_{2}\left(\mathbf{1}-p\right), \ldots,pc_{n}\left(\mathbf{1}-p\right)  \right\}$ in $pA$ such that the sets together satisfy properties (i) to (iii) in Theorem \ref{b3.7}. We shall use these sets to construct $S'$ and $T'$: To this end, set $S' := S$ and $T' := \left\{p, pc_{2}\left(\mathbf{1}-p\right), \ldots,pc_{n}\left(\mathbf{1}-p\right)  \right\}$. Consider the element $pb_{2}\left(\mathbf{1}-p\right)$. If $pb_{2}\left(\mathbf{1}-p\right) \in \mathrm{span}\;T'$, then we leave $S'$ and $T'$ unchanged. On the other hand, if $pb_{2}\left(\mathbf{1}-p\right) \notin \mathrm{span}\;T'$, then we proceed with the following scheme: For each $i \in \left\{2, \ldots, n \right\}$, let $\alpha_{i} = - \mathrm{Tr}\,\left(pb_{2}\left(\mathbf{1}-p\right) a_{i}p\right)$. Let $c_{n+1} = b_{2} + \sum_{i = 2}^{n} \alpha_{i}c_{i}$ and notice that $pc_{n+1}\left(\mathbf{1}-p\right) \notin \mathrm{span}\;T'$. Thus, in particular, the set $\left\{p, pc_{2}\left(\mathbf{1}-p\right), \ldots,pc_{n+1}\left(\mathbf{1}-p\right) \right\}$ is linearly independent. Moreover, for any $j \in \left\{2, \ldots, n \right\}$ we have that
\begin{eqnarray*}
pc_{n+1}\left(\mathbf{1}-p\right)a_{j}p & = & p\left(b_{2} + \sum_{i = 2}^{n} \alpha_{i}c_{i}\right)\left(\mathbf{1}-p\right)a_{j}p \\
& = & pb_{2}\left(\mathbf{1}-p\right)a_{j}p + \alpha_{j}p \\
& = & \mathrm{Tr}\,\left(pb_{2}\left(\mathbf{1}-p\right)a_{j}p\right) p + \alpha_{j} p = 0.
\end{eqnarray*}
Next we apply Theorem \ref{b3.3} and the Spectral Mapping Theorem to obtain an element $a_{n+1} \in A$ such that $pc_{i}\left(\mathbf{1}-p\right)a_{n+1}p = 0$ for each $i \in \left\{2, \ldots, n \right\}$ and $pb_{2}\left(\mathbf{1}-p\right)a_{n+1}p = p$. We claim that the set $\left\{p, \left(\mathbf{1}-p\right)a_{2}p, \ldots,\left(\mathbf{1}-p\right)a_{n+1}p \right\}$ is linearly independent: Assume that
\begin{equation}
\lambda_{1}p + \lambda_{2}\left(\mathbf{1}- p\right)a_{2}p + \cdots + \lambda_{n+1}\left(\mathbf{1}- p\right)a_{n+1}p = 0
\label{beq9}
\end{equation}
for some $\lambda_{1}, \ldots, \lambda_{n} \in \mathbb{C}$. Let $i \in \left\{2, \ldots, n+1 \right\}$ be arbitrary but fixed. Multiplying both sides of (\ref{beq9}) on the left by $pc_{i}\left(\mathbf{1}-p\right)$ yields $\lambda_{i}p = 0$. Consequently, we may conclude that $\lambda_{i} = 0$ for each $i \in \left\{2, \ldots, n+1 \right\}$. Thus, (\ref{beq9}) becomes $\lambda_{1}p = 0$ which gives $\lambda_{1} = 0$ and proves our claim. We now set $S':= \left\{p, \left(\mathbf{1}-p\right)a_{2}p, \ldots,\left(\mathbf{1}-p\right)a_{n+1}p  \right\}$ and $T' = \left\{p, pc_{2}\left(\mathbf{1}-p\right), \ldots,pc_{n+1}\left(\mathbf{1}-p\right) \right\}$. A quick inspection shows that $S'$ and $T'$ satisfy properties (i) to (iii) from Theorem \ref{b3.7}. Next we consider the element $pb_{3}\left(\mathbf{1}-p\right)$ and repeat the iteration scheme above. After a total of $m-1$ iterations, we obtain the desired sets $S'$ and $T'$.
\end{proof}

For any $z_{1}, \ldots, z_{r} \in A$, we denote by $C\left[z_{1}, \ldots, z_{r}\right]$ the algebra generated by $z_{1}, \ldots, z_{r}$, that is,
$$C\left[z_{1}, \ldots, z_{r}\right] = \left\{q\left(z_{1}, \ldots, z_{r}\right) : q\;\,\mathrm{is\;a\;polynomial\;without\;constant\;term} \right\}.$$ 

\begin{lemma}\label{b3.12}
Let $p$ be a rank one projection and let $z_{1}, \ldots, z_{r} \in J_{p}$. Then there exists a subalgebra $B$ of $J_{p}$ such that $z_{1},\ldots, z_{r} \in B \cong M_{k} \left(\mathbb{C}\right)$. Moreover, $C\left[z_{1}, \ldots, z_{r}\right] \subseteq B$ and $z_{j}Az_{j} \subseteq B$ for each $j \in \left\{1, \ldots, r \right\}$.
\end{lemma}

\begin{proof}
If $\mathrm{dim}\:Ap$ or $\mathrm{dim}\:pA$ is finite, then the result follows from Lemma \ref{b3.8}. So we may assume that $\mathrm{dim}\:Ap = \mathrm{dim}\:pA = \infty$. Let $V$ and $W$ be bases for $Ap$ and $pA$, respectively, both containing $p$. By Lemma \ref{b3.2} we may assume that any element $y \in V-\left\{p \right\}$ is of the form $y = \left(\mathbf{1}-p\right)xp$ for some $x \in A$. A similar observation holds for $W$. Now, for any $j \in \left\{1, \ldots, r \right\}$ it follows that
$z_{j} = \sum_{i = 1}^{N_{j}} \alpha_{j, i}x_{j, i}py_{j, i}$, where $\alpha_{j, 1}, \ldots, \alpha_{j, N_{j}} \in \mathbb{C}$, $x_{j, 1}, \ldots, x_{j, N_{j}} \in V$ and $y_{j, 1}, \ldots, y_{j, N_{j}} \in W$. To simplify our notation we may assume that
\begin{eqnarray*}
S & := & \left\{p\right\} \cup \left\{x_{j, 1}, \ldots, x_{j, N_{j}} : j \in \left\{1, \ldots, r\right\} \right\} \\
& = & \left\{p, \left(\mathbf{1}-p\right)a_{2}p, \ldots, \left(\mathbf{1}-p\right)a_{n}p  \right\},
\end{eqnarray*}
and that
\begin{eqnarray*} 
T & := & \left\{p \right\} \cup \left\{y_{j, 1}, \ldots, y_{j, N_{j}} : j \in \left\{1, \ldots, r\right\} \right\} \\
& = & \left\{p, pb_{2}\left(\mathbf{1}-p\right), \ldots,pb_{m}\left(\mathbf{1}-p\right)  \right\}.
\end{eqnarray*} 
By Lemma \ref{b3.11} there are two linearly independent subsets $S'$ and $T'$ of $Ap$ and $pA$, respectively, such that $\mathrm{span}\;S \subseteq \mathrm{span}\;S'$, $\mathrm{span}\;T \subseteq \mathrm{span}\;T'$ and properties \textnormal{(i)} to \textnormal{(iii)} in Theorem \textnormal{\ref{b3.7}} holds for $S'$ and $T'$. Say $S' = \left\{p, u_{2}, \ldots, u_{k} \right\}$ and $T' = \left\{p, v_{2}, \ldots, v_{k} \right\}$. If we take
$$B = \mathrm{span}\;\left(\left\{p \right\} \cup \left\{u_{i}pv_{j}, u_{i}p, pv_{j} : i, j \in \left\{2, \ldots, k \right\}\right\}\right),$$
then $B$ is the desired subalgebra of $J_{p}$. Finally, we note that the containments $C\left[z_{1}, \ldots, z_{r}\right] \subseteq B$ and $z_{j}Az_{j} \subseteq B$ for each $j \in \left\{1, \ldots, r \right\}$ follows from properties (i) to (iii) for $S'$ and $T'$ and the minimality of $p$. 
\end{proof}

\begin{theorem}\label{b3.13}
Let $z_{1}, \ldots, z_{r} \in \mathrm{Soc}\:A$. Then there exists a subalgebra $B$ of $\mathrm{Soc}\:A$ such that
$$z_{1}, \ldots, z_{r} \in B \cong M_{n_{1}}\left(\mathbb{C}\right) \oplus \cdots \oplus M_{n_{k}}\left(\mathbb{C}\right).$$
Moreover, $C\left[z_{1}, \ldots, z_{r}\right] \subseteq B$ and $z_{j}Az_{j} \subseteq B$ for each $j \in \left\{1, \ldots, r \right\}$. 
\end{theorem}

\begin{proof}
This is an immediate consequence of Lemma \ref{b2.1} and Lemma \ref{b3.12}.
\end{proof}

As in \cite{finiterankelements}, we denote by $M_{r, n}$, where $r \leq n \leq 2r$, the algebra of $n \times n$ matrices $S=\left[\alpha_{ij}\right]$ satisfying $\alpha_{ij} = 0$ whenever $i > r$ or $j \leq n-r$. In \cite[Lemma 2.7]{finiterankelements}, Bre\v{s}ar and \v{S}emrl managed to prove that an operator $T \in B(X)$ has rank $r$ if and only if $TB(X)T \cong M_{r, n}$ for some $n$. Moreover, by \cite[Theorem 3.8, Theorem 3.13]{tracesocleident} it follows that $\mathrm{Soc}\:B(X)$ is a minimal two-sided ideal. Thus, although the structure of $TB(X)T$ may be very complicated, by Lemma \ref{b2.2} and Lemma \ref{b3.12} it is possible to find a full matrix algebra in which $TB(X)T$ can be embedded.     

\section{Commutators in the Socle}

For the convenience of the reader we recall a few results from \cite{tracesocleident}:

\begin{lemma}\cite[Lemma 3.2]{tracesocleident}\label{b4.1}
Let $x \in pAp = B$, where $p = p_{1} + \cdots + p_{n}$ with $p_{1}, \ldots, p_{n}$ orthogonal rank one projections of $A$. Then $\mathrm{Tr}_{A} (x) = \mathrm{Tr}_{B} (x)$.
\end{lemma}

\begin{theorem}\cite[Theorem 3.4]{tracesocleident}\label{b4.2}
Suppose that $\mathrm{Soc}\:A \neq \left\{0 \right\}$. For any linear functional $f$ on $\mathrm{Soc}\:A$ we have that $f = \alpha\mathrm{Tr}$ for some $\alpha \in \mathbb{C}$ if and only if $f$ is constant on the rank one projections of $A$.
\end{theorem}

\begin{theorem}\cite[Theorem 3.8, Theorem 3.9]{tracesocleident}\label{b4.3}
The following are equivalent:
\begin{itemize}
\item[\textnormal{(a)}]
For any linear functional $f$ on $\mathrm{Soc}\:A$ we have that
$$f(ab) = f(ba) \;\,\mathrm{for\;all}\;\, a, b \in \mathrm{Soc}\:A \Leftrightarrow f = \alpha \mathrm{Tr} \;\,\mathrm{for\;some}\;\, \alpha \in \mathbb{C}.$$
\item[\textnormal{(b)}]
$\mathrm{Soc}\:A$ is a minimal two-sided ideal.
\item[\textnormal{(c)}]
$pAp \cong M_{n_{p}} \left(\mathbb{C}\right)$ for each finite-rank projection $p$ of $A$.
\end{itemize}
\end{theorem}

If $x, y \in A$, then we define the \textit{commutator} of $x, y$ by
$$\left[x, y\right] = xy-yx.$$
Shoda's Theorem (see \cite{shodathm}) says that if $A = M_{k} \left(\mathbb{C}\right)$, then the traceless matrices in $A$ are precisely those matrices that can be expressed as commutators. As a particular example, we obtain the classical fact that this result can be generalized to $B(X)$. However, Shoda's Theorem fails in general. We can in fact give a precise characterization of those Banach algebras in which Shoda's Theorem holds:

\begin{theorem}\label{b4.4}\textnormal{(Generalized Shoda's Theorem.)}\\
Every traceless element $a \in \mathrm{Soc}\:A$ can be expressed as the commutator of two elements belonging to $\mathrm{Soc}\:A$ if and only if $\mathrm{Soc}\:A$ is a minimal two-sided ideal. Moreover, the rank of each of the two elements in the commutator does not exceed $\mathrm{rank}\,(a)$. In particular, the Generalized Shoda's Theorem is valid for $A = B(X)$.
\end{theorem}

\begin{proof}
Suppose that $\mathrm{Soc}\:A$ is a minimal two-sided ideal of $A$, and suppose that $a \in \mathrm{Soc}\:A$ has $\mathrm{Tr}\,(a)=0$. If $a = 0$, then $a$ is a commutator. So assume that $a \neq 0$. By the density of $E(a)$ and the Diagonalization Theorem there exist a $u \in G(A)$, orthogonal rank one projections $p_{1}, \ldots, p_{n}$ and $\lambda_{1}, \ldots, \lambda_{n} \in \mathbb{C}-\left\{0 \right\}$ such that $a = \lambda_{1}up_{1} + \cdots + \lambda_{n}up_{n}$. So, by the orthogonality of the $p_{i}$, $a$ can be expressed as $a=ap$, where $p = p_{1} + \ldots + p_{n}$ is a finite-rank projection. Setting $r_{a} = ap-pap$, write
$$a = pap + \left(ap - pap\right) = pap + r_{a}$$
and notice that $pr_{a} = 0$ and $r_{a}p = r_{a}$. Observe also that $\mathrm{Tr}\,(pap) = \mathrm{Tr}\,(ap) = \mathrm{Tr}\,(a) = 0$. Now, by Theorem \ref{b4.3} it follows that $pAp \cong M_{k}\left(\mathbb{C}\right)$. Thus, by Lemma \ref{b4.1} and the classical Shoda's Theorem for matrices, we have $pap = \left[pxp, pyp\right]$ for some $x, y \in A$. Finally pick $\left|\lambda\right|$ sufficiently large so that $\lambda p + pyp \in G\left(pAp\right)$. Then
\begin{eqnarray*}
&&\left[pxp+r_{a}\left(\lambda p +pyp\right)^{-1}, \lambda p + pyp\right] \\ & = & (pxp)(pyp)+\lambda pxp + r_{a}p - (pyp)(pxp) - \lambda pxp \\
& = & \left[pxp, pyp\right] + r_{a}p  =  pap +r_{a}p = a.
\end{eqnarray*}
Since
$$\mathrm{rank}\,(p) = \mathrm{rank}\,\left(p_{1}\right) + \cdots + \mathrm{rank}\,\left(p_{n}\right) = n =  \mathrm{rank}\,(a),$$ 
it follows from the properties of the rank that the rank of each of the two elements in the commutator does not exceed $\mathrm{rank}\,(a)$. This proves the reverse implication. For the forward implication, let $f$ be any linear functional on $\mathrm{Soc}\:A$ such that $f(ab) = f(ba)$ for all $a, b \in \mathrm{Soc}\:A$. Now, for any rank one projections $p$ and $q$ of $A$ we have $\mathrm{Tr}\,(p-q) = \mathrm{Tr}\,(p) - \mathrm{Tr}\,(q) = 0$. Hence, $p - q$ is a commutator. Therefore, by the linearity of $f$ we may conclude that $f(p) = f(q)$. In fact, since $p$ and $q$ were arbitrary, it follows that $f$ is constant on the rank one projections of $A$. Thus, by Theorem \ref{b4.2} it follows that $f = \alpha \mathrm{Tr}$ for some $\alpha \in \mathbb{C}$. By Theorem \ref{b4.3} this shows that $\mathrm{Soc}\:A$ is indeed a minimal two-sided ideal. This proves the forward implication. The last observation is true since $\mathrm{Soc}\:B(X)$ is a minimal two-sided ideal.
\end{proof}

Of course it is possible to obtain the reverse implication of Theorem \ref{b4.4} by means of Lemma \ref{b2.2} and Lemma \ref{b3.12}. However, this approach does not yield the upper bound for the rank of each of the two elements in the commutator.\\

By \cite[Theorem 2.12]{aupetitmoutontrace} it follows that
$$\mathrm{rank}\,(a) \leq \mathrm{dim}\:aAa \leq \left[\mathrm{rank}\,(a)\right]^{2} \;\,\mathrm{for\;all}\;\,a\in A.$$
We also have the following:

\begin{theorem}\label{b4.5}
Suposse that $a \in \mathrm{Soc}\:A$ and that $\mathrm{Tr}\,(a) = 0$. If $\mathrm{dim}\:aAa = \left[\mathrm{rank}\,(a)\right]^{2}$, then $a = \left[x, y\right]$ for some $x, y \in \mathrm{Soc}\:A$.
\end{theorem} 

\begin{proof}
If $a = 0$, then the result is obviously true. So assume that $a \neq 0$. By the density of $E(a)$ and the Diagonalization Theorem there exist a $u \in G(A)$, orthogonal rank one projections $p_{1}, \ldots, p_{n}$ and $\lambda_{1}, \ldots, \lambda_{n} \in \mathbb{C}-\left\{0 \right\}$ such that $a = \lambda_{1}up_{1} + \cdots + \lambda_{n}up_{n}$, where $\mathrm{rank}\,(a) = n \geq 1$. Let $p = p_{1} + \cdots +p_{n}$ and observe that $ap = a$ and that
$$a \left(\frac{1}{\lambda_{1}}p_{1} + \cdots + \frac{1}{\lambda_{n}}p_{n}\right) = up.$$
Let $w \in A$ be arbitrary and consider $awa$. Then
\begin{eqnarray*}
awa = apwap & = & a\left(\frac{1}{\lambda_{1}}p_{1} + \cdots + \frac{1}{\lambda_{n}}p_{n}\right)\left(\lambda_{1}p_{1} + \cdots +\lambda_{n}p_{n}\right) wap \\
& = & up\left(\lambda_{1}p_{1} + \cdots +\lambda_{n}p_{n}\right) wap.
\end{eqnarray*}
This shows that $aAa \subseteq upAp$. Conversely, if we let $v \in A$ be arbitrary, then
\begin{eqnarray*}
upvp  & = &  u\left(\lambda_{1}p_{1} + \cdots +\lambda_{n}p_{n}\right)\left(\frac{1}{\lambda_{1}}p_{1} + \cdots + \frac{1}{\lambda_{n}}p_{n}\right)vp \\
& = & a\left(\frac{1}{\lambda_{1}}p_{1} + \cdots + \frac{1}{\lambda_{n}}p_{n}\right)v\left(\frac{1}{\lambda_{1}}p_{1} + \cdots + \frac{1}{\lambda_{n}}p_{n}\right)u^{-1}u\left(\lambda_{1}p_{1} + \cdots +\lambda_{n}p_{n}\right) \\
& = & a\left(\frac{1}{\lambda_{1}}p_{1} + \cdots + \frac{1}{\lambda_{n}}p_{n}\right)v\left(\frac{1}{\lambda_{1}}p_{1} + \cdots + \frac{1}{\lambda_{n}}p_{n}\right)u^{-1}a.
\end{eqnarray*}
This shows that $upAp \subseteq aAa$. Hence, equality holds. In particular, this implies that
$$\mathrm{dim}\,\left(pAp\right) = \mathrm{dim}\,\left(upAp\right) = \mathrm{dim}\,\left(aAa\right) = \left[\mathrm{rank}\,(a)\right]^{2}.$$
Moreover, we have
$$\mathrm{rank}\,(p) = \mathrm{rank}\,\left(p_{1} + \cdots + p_{n}\right) = \sum_{i=1}^{n} \mathrm{rank}\,\left(p_{i}\right) = n = \mathrm{rank}\,(a).$$
Now, recall that $pAp$ is a finite-dimensional closed semisimple subalgebra of $A$ with identity $p$. Also, $\mathrm{rank}_{pAp} (pxp) = \mathrm{rank}_{A} (pxp)$ for all $x \in A$. By the Wedderburn-Artin Theorem, it follows that
$$pAp \cong M_{n_{1}} \left(\mathbb{C}\right) \oplus \cdots \oplus M_{n_{k}} \left(\mathbb{C}\right),$$
where we may assume without loss of generality that $k \geq 1$ and $n_{i} \geq 1$ for each $i \in \left\{1, \ldots, k \right\}$. We claim that $k=1$: Suppose not. We have $n = \mathrm{rank}\,(p) = n_{1} + \cdots +n_{k}$ and $n^{2} = \mathrm{dim}\:pAp = n_{1}^{2} + \cdots +n_{k}^{2}$. However, then
\begin{eqnarray*}
n^{2} = \left(n_{1} + \cdots + n_{k}\right)^{2} & = & n_{1}^{2} + \cdots +n_{k}^{2} + \sum_{i<j}2n_{i}n_{j} \\
& > & n_{1}^{2} + \cdots +n_{k}^{2} = \mathrm{dim}\:pAp = n^{2},
\end{eqnarray*}
which is clearly a contradiction. Hence, $pAp \cong M_{n_{1}} \left(\mathbb{C}\right)$ for some integer $n_{1} \geq 1$. Using the argument in the proof of Theorem \ref{b4.4}, it readily follows that $a = pap + r_{a}$ is a commutator as desired.
\end{proof}

\begin{example}\label{b4.6}
Let $A = M_{2}\left(\mathbb{C}\right) \oplus M_{2}\left(\mathbb{C}\right)$ and let
$$a = \left( \left(\begin{array}{cc} 1 & 0 \\ 0 & -1\end{array}\right), \left(\begin{array}{cc} 1 & 0 \\ 0 & -1\end{array}\right) \right).$$
Then $\mathrm{rank}\,(a) = 4$ and by the classical Shoda's Theorem for matrices it follows that $a = \left[x, y\right]$ for some $x, y \in A = \mathrm{Soc}\:A$. However, $aAa = A$, so 
$$\mathrm{dim}\:aAa = 8 \neq 16 = \left[\mathrm{rank}\,(a)\right]^{2}.$$
This shows that the converse of Theorem \ref{b4.5} does not hold in general.
\end{example}

The aim of the next few results is to prove the highly nontrivial fact that
$$\mathcal{C} = \left\{\left[x, y\right]: x, y \in \mathrm{Soc}\:A \right\}$$
is a vector subspace of $\mathrm{Soc}\:A$.

\begin{lemma}\label{b4.7}
Let $p$ be a rank one projection of $A$. If $q \in J_{p} - \left\{0 \right\}$ is a projection, then $qAq \cong M_{k} \left(\mathbb{C}\right)$ for some integer $k \geq 1$.
\end{lemma}

\begin{proof}
By the Wedderburn-Artin Theorem,
\begin{equation}
qAq \cong M_{n_{1}}\left(\mathbb{C}\right) \oplus \cdots \oplus M_{n_{k}}\left(\mathbb{C}\right).
\label{beq11}
\end{equation}
Assume that the direct sum in (\ref{beq11}) contains at least two nonzero terms. This means that we can find distinct rank one projections $r$ and $s$ such that $rxs = 0$ and $sxr = 0$ for all $x \in A$. Since $q \in J_{p}$, it follows that $r \in J_{p}$. Consequently, $\left\{0 \right\} \neq J_{r} \subseteq J_{p}$. Thus, by Lemma \ref{b2.2} it follows that $J_{r} = J_{p}$. Similarly, $J_{s} = J_{p}$. However, then $r = \sum_{i=1}^{m} x_{i}sy_{i}$ for some $x_{1}, \ldots, x_{m}, y_{1}, \ldots, y_{m} \in A$ and integer $m \geq 1$. But this implies that
$$r = r^{2} = \sum_{i=1}^{m} rx_{i}sy_{i} = 0$$
which is absurd. From this contradiction, the result now follows.
\end{proof}

\begin{lemma}\label{b4.8}
Let $p$ be a rank one projection. If $a \in J_{p}$ and $\mathrm{Tr}\,(a) = 0$, then $a = \left[x, y\right]$ for some $x, y \in J_{p}$. Hence, if $b, c \in \mathcal{C} \cap J_{p}$, then $b+c \in \mathcal{C}$. In particular, $b+c = \left[u, v\right]$ for some $u, v \in J_{p}$.
\end{lemma}

\begin{proof}
If $a = 0$, then the result trivially holds true. So assume that $a \neq 0$. By the density of $E(a)$ and the Diagonalization Theorem there exist a $u \in G(A)$, mutually orthogonal rank one projections $q_{1}, \ldots, q_{n} \in Aa$ and $\lambda_{1}, \ldots, \lambda_{n} \in \mathbb{C}-\left\{0 \right\}$ such that $a = \lambda_{1}uq_{1} + \cdots +\lambda_{n}uq_{n}$, where $\mathrm{rank}\,(a) = n\geq 1$. Let $q = q_{1} + \cdots +q_{n}$. Then $aq = a$, and so, by the properties of the trace and Lemma \ref{b4.1},
$$ 0 = \mathrm{Tr}_{A}(a) = \mathrm{Tr}_{A}(qaq) = \mathrm{Tr}_{qAq}(qaq).$$
Moreover, by Lemma \ref{b4.7} it follows that $qAq \cong M_{k} \left(\mathbb{C}\right)$. So, by the classical Shoda's Theorem for matrices, we have $qaq = \left[qxq, qyq\right]$ for some $x, y \in A$. Consequently, by the argument used in the proof of Theorem \ref{b4.4} with $\left|\lambda\right|$ sufficiently large we get
$$a = \left[qxq + r_{a}\left(\lambda q + qyq\right)^{-1}, \lambda q + qyq\right],$$
where $r_{a} = aq - qaq$. Of course, since $q \in Aa\subseteq J_{p}$, it follows that $qxq + r_{a}\left(\lambda q + qyq\right)^{-1}$ and $\lambda q + qyq$ both belong to $J_{p}$. This proves the first part of the lemma. The last part now follows from the fact that $b, c \in \mathcal{C} \cap J_{p}$ implies that $b+c \in J_{p}$ and $\mathrm{Tr}\,(b+c) = 0$.
\end{proof}

\begin{theorem}\label{b4.9}
$\mathcal{C} = \left\{\left[x, y\right] : x, y \in \mathrm{Soc}\:A \right\}$ is a vector subspace of $\mathrm{Soc}\:A$.
\end{theorem}

\begin{proof}
If $a \in \mathcal{C}$, then it readily follows that $\lambda a \in \mathcal{C}$ for all $\lambda \in \mathbb{C}$. So it remains to show that if $a, b \in \mathcal{C}$, then $a+b \in \mathcal{C}$: Let $a, b \in \mathcal{C}$ be arbitrary. Let $\mathcal{P}$ be the class of minimal projections from Lemma \textnormal{\ref{b2.1}} generating the $J_{p}$. By Lemma \ref{b2.1} it follows that
$$a = \left[w_{1}, x_{1}\right] + \cdots + \left[w_{n}, x_{n}\right],$$
where $w_{i}, x_{i} \in J_{q_{i}}$ and $q_{i} \in \mathcal{P}$ for each $i \in \left\{1, \ldots, n \right\}$, and where $q_{i} \neq q_{j}$ for $i \neq j$. Similarly,
$$b = \left[y_{1}, z_{1}\right] + \cdots + \left[y_{k}, z_{k}\right],$$
where $y_{i}, z_{i} \in J_{p_{i}}$ and $p_{i} \in \mathcal{P}$ for each $i \in \left\{1, \ldots, k \right\}$, and where $p_{i} \neq p_{j}$ for $i \neq j$. Consequently,
$$a + b = \left[w_{1}, x_{1}\right] + \cdots + \left[w_{n}, x_{n}\right] + \left[y_{1}, z_{1}\right] + \cdots + \left[y_{k}, z_{k}\right].$$
Now, if $q_{i} = p_{j}$ for some $i \in \left\{1, \ldots, n \right\}$ and $j \in \left\{1, \ldots, k \right\}$, then by Lemma \ref{b4.8} it follows that
$$\left[w_{i}, x_{i}\right] + \left[y_{j}, z_{j}\right] = \left[x, y\right]$$
for some $x, y \in J_{p_{i}} = J_{q_{j}}$. So, in order to simplify our notation we may assume, without loss of generality, that $q_{i} \neq p_{j}$ for all $i \in \left\{1, \ldots, n \right\}$ and $j \in \left\{1, \ldots, k \right\}$. Thus, using the fact that the two-sided ideals in $\left\{J_{p} : p \in \mathcal{P} \right\}$ are pairwise orthogonal, we get
$$a+b = \left[w_{1} + \cdots +w_{n} +y_{1} + \cdots +y_{k}, x_{1} + \cdots +x_{n} +z_{1} + \cdots +z_{k}\right].$$
The elements in this commutator are of course in $\mathrm{Soc}\:A$, and so, $a+b \in \mathcal{C}$. This completes the proof. 
\end{proof}

\begin{corollary}\label{b4.10}
Let $a \in \mathcal{C}$ and let $p$ be any finite-rank projection such that $ap = a$. Then $pap \in \mathcal{C}$.
\end{corollary}

\begin{proof}
Simply observe that $pap = pap - ap + a$ and that $pap -ap = \left[pap - ap, p\right]$. Now apply Theorem \ref{b4.9}.
\end{proof}

\begin{corollary}\label{b4.11}
Let $\mathcal{P}$ be the class of minimal projections from Lemma \textnormal{\ref{b2.1}} generating the $J_{p}$. Then $a \in \mathcal{C}$ if and only if $a = a_{1} + \cdots + a_{n}$, where $a_{i} \in J_{p_{i}}$, $p_{i} \in \mathcal{P}$ and $\mathrm{Tr}_{A} \left(a_{i}\right) = 0$ for each $i \in \left\{1, \ldots, n \right\}$.
\end{corollary}

\begin{proof}
The forward implication follows from Lemma \ref{b2.1} since
$$a = \left[w_{1}, x_{1}\right] + \cdots + \left[w_{n}, x_{n}\right],$$
where $w_{i}, x_{i} \in J_{p_{i}}$, $p_{i} \in \mathcal{P}$ for each $i \in \left\{1, \ldots, n \right\}$ and $p_{i} \neq p_{j}$ for $i \neq j$. The reverse implication follows from Lemma \ref{b4.8} and Theorem \ref{b4.9}.
\end{proof}

\begin{lemma}\label{b4.12}
Let $a \in \mathrm{Soc}\:A$ and suppose that $\sigma (a) = \left\{0 \right\}$. Then $a \in \mathcal{C}$.
\end{lemma}

\begin{proof}
If $a = 0$, then the result is obviously true. So assume that $a \neq 0$. By the density of $E(a)$ and the Diagonalization Theorem there exist a $u \in G(A)$, orthogonal rank one projections $p_{1}, \ldots, p_{n}$ and $\lambda_{1}, \ldots, \lambda_{n} \in \mathbb{C}-\left\{0 \right\}$ such that $a = \lambda_{1}up_{1} + \cdots + \lambda_{n}up_{n}$, where $\mathrm{rank}\,(a) = n \geq 1$. Let $p = p_{1} + \cdots +p_{n}$ and observe that $ap = a$. Hence, by Jacobson's Lemma it follows that $\left\{0 \right\} = \sigma_{A} (pap) = \sigma_{pAp} (pap)$. Now, by the Wedderburn-Artin Theorem it follows that $pAp$ is isomorphic as an algebra to
$$M = M_{n_{1}} \left(\mathbb{C}\right) \oplus \cdots \oplus M_{n_{1}} \left(\mathbb{C}\right).$$
Let $\phi : pAp \rightarrow M$ be the algebra isomorphism, and let $M_{i} = M_{n_{i}} \left(\mathbb{C}\right)$ for each $i \in \left\{1, \ldots, k \right\}$. We claim that $pap \in \mathcal{C}$: To prove our claim it will suffice to show that $\phi (pap) = \left(a_{1}, \ldots, a_{k}\right)$ is a commutator in $M$. Since
$$ \left\{0 \right\} = \sigma_{pAp} (pap) = \sigma_{M} \left(\phi(pap)\right) = \bigcup_{i=1}^{k} \sigma_{M_{i}} \left(a_{i}\right),$$
it readily follows that $a_{i}$ is a nilpotent and hence traceless element of $M_{i}$ for each $i \in \left\{1, \ldots, k \right\}$. So by Shoda's Theorem for matrices, it follows that $a_{i}$ is a commutator of $M_{i}$ for each $i \in \left\{1, \ldots, k \right\}$. By the pointwise definition of multiplication in the direct sum, this proves our claim. It is now possible to proceed as in the proof of Theorem \ref{b4.4} and conclude that $a = pap +r_{a} \in \mathcal{C}$. So the lemma is proved. 
\end{proof}

\begin{theorem}\label{b4.13}
Let $a \in \mathrm{Soc}\:A$ with $\mathrm{Tr}\,(a) = 0$, and let $p$ be the Riesz projection associated with $\sigma'(a)$ and $a$. If $\mathrm{dim}\:pAp = \left[\mathrm{rank}\,(p)\right]^{2}$, then $a \in \mathcal{C}$. 
\end{theorem}

\begin{proof}
By the Holomorphic Functional Calculus it follows that $p \in Aa$, $ap = pa = pap$ and that $\sigma \left(\left(\mathbf{1}-p\right)a\right) = \left\{0 \right\}$. Thus, in particular, $\mathrm{Tr}\,\left(\left(\mathbf{1}-p\right)a\right) = 0$. Hence, by the linearity of the trace it follows that $\mathrm{Tr}\,(pap) = 0$. Now, since $\mathrm{dim}\:pAp = \left[\mathrm{rank}\,(p)\right]^{2}$, it follows from the argument used in the proof of Theorem \ref{b4.5} that $pAp \cong M_{n}\left(\mathbb{C}\right)$. Hence, by Lemma \ref{b4.1} and the classical Shoda's Theorem for matrices that $pap = \left[pxp, pyp\right]$ for some $x, y \in A$. Moreover, since  $\left(\mathbf{1}-p\right)a \in \mathrm{Soc}\:A$ and $\sigma \left(\left(\mathbf{1}-p\right)a\right) = \left\{0 \right\}$, it follows from Lemma \ref{b4.12} that $\left(\mathbf{1}-p\right)a \in \mathcal{C}$. Therefore, since $\mathcal{C}$ is a vector space by Theorem \ref{b4.9} and $a = \left(\mathbf{1}-p\right)a + pap$, it follows that $a \in \mathcal{C}$ as desired. 
\end{proof}

\section{Characterizations of Central Socles}

If $\mathrm{Soc}\:A = \left\{0 \right\}$, then obviously $\mathrm{Soc}\:A \subseteq Z(A)$. For this reason we shall assume throughout this section that $\mathrm{Soc}\:A \neq \left\{0 \right\}$.

\begin{lemma}\label{b5.1}
$\mathrm{Soc}\:A \subseteq Z(A)$ if and only if 
\begin{equation}
xp = x \Leftrightarrow px = x
\label{beq12}
\end{equation}
for all $x \in \mathrm{Soc}\:A$ and rank one projections $p$ of $A$.
\end{lemma}

\begin{proof}
The forward implication is obvious. Assume next that $\mathrm{Soc}\:A \not\subseteq Z(A)$. Then there exists a rank one projection $p$ such that $ p \notin Z(A)$. Hence, $xp \neq px$ for some $x \in A$. However, if (\ref{beq12}) holds, then 
$$xp^{2} = xp \Rightarrow pxp = xp \;\, \mathrm{and}\;\, p^{2}x = px \Rightarrow pxp = px.$$ 
But then $xp=px$ which is absurd. So (\ref{beq12}) dos not hold. This completes the proof.
\end{proof}

\begin{lemma}\label{b5.2}
$\mathrm{Soc}\:A \subseteq Z(A)$ if and only if for all rank one projections $p$ and $q$ of $A$, $\mathrm{dim}\:pAq = 1 \Rightarrow p = \alpha q$ for some $\alpha \in \mathbb{C}$.
\end{lemma}

\begin{proof}
Suppose that $\mathrm{Soc}\:A \subseteq Z(A)$. Let $p$ and $q$ be any rank one projections of $A$ such that $\mathrm{dim}\:pAq = 1$. Then, in particular, $pq=qp\neq 0$. Hence, by the minimality of $p$ and $q$ it follows that $pq = qpq = \lambda q$ and $pq = pqp = \beta p$ for some $\lambda, \beta \in \mathbb{C}$. Thus, $p = \frac{\lambda}{\beta}q$. For the converse, suppose that $\mathrm{dim}\:pAq = 1$ implies that $p = \alpha q$ for some $\alpha \in \mathbb{C}$ for all rank one projections $p$ and $q$ of $A$ but that $\mathrm{Soc}\:A \not\subseteq Z(A)$. By Lemma \ref{b5.1} there exist an $x \in \mathrm{Soc}\:A$ and a rank one projection $p$ such that either $xp = x$ and $px \neq x$, or, $px = x$ and $xp \neq x$. Say it is the former. In particular, $x \neq 0$. Moreover, since $xp=x$, it follows that $\mathrm{rank}\,(x) = 1$. Thus, by the density of $E(x)$ and the Diagonalization Theorem, there exist a $u \in G(A)$ and a rank one projection $q$ such that $x = qu$. Hence, $x = xp =qup$, and so, $0 \neq x \in qAp$. It can be shown that $\mathrm{dim}\:qAp \leq 1$ (see \cite[Lemma 4.2]{puhltrace}), so it must be the case that $\mathrm{dim}\:qAp = 1$. Hence, by hypothesis, $q = \lambda p$ for some $\lambda \in \mathbb{C}-\left\{0 \right\}$. However, then
$$px = p \left(qu\right)= p \left(\lambda pu\right) = \lambda pu = qu =x$$
which is a contradiction. Thus, it must be the case that $\mathrm{Soc}\:A \subseteq Z(A)$. This gives the result.
\end{proof}

\begin{lemma}\label{b5.3}
$\mathrm{Soc}\:A \subseteq Z(A)$ if and only if for all rank one projections $p$ and $q$ of $A$, $pq = 0 \Rightarrow \mathrm{dim}\:pAq = 0$ and $pq \neq 0 \Rightarrow p = \alpha q$ for some $\alpha \in \mathbb{C}$.
\end{lemma}

\begin{proof}
The forward implication is obvious. For the reverse implication, let $p$ and $q$ be any rank one projections of $A$ such that $\mathrm{dim}\:pAq = 1$. By hypothesis, $pq \neq 0$ and hence $p = \alpha q$ for some $\alpha \in \mathbb{C}$. Thus, by Lemma \ref{b5.2} we have the result.
\end{proof}

\begin{theorem}\label{b5.4}
$\mathrm{Soc}\:A \subseteq Z(A)$ if and only if for all $x \in \mathrm{Soc}\:A$, $x^{2} = 0 \Rightarrow x=0$.
\end{theorem}

\begin{proof}
Suppose first that $\mathrm{Soc}\:A \subseteq Z(A)$. Let $x \in \mathrm{Soc}\:A$ such that $x^{2}= 0$. Since $x$ is von Neumann regular, it follows that $x = xyx$ for some $y \in A$. Thus, $x = x^{2}y = 0$. Conversely, suppose that $x=0$ whenever $x^{2} = 0$ for all $x \in \mathrm{Soc}\:A$. Let $p$ and $q$ be any rank one projections such that $pq = 0$. Then, in particular, $\left(qp\right)^{2} = qpqp = 0$. Thus, $qp = 0$. Consequently, for any $y \in A$, we have $\left(pyq\right)^{2} = pyqpyq = 0$, and so, $pyq = 0$. In other words, $\mathrm{dim}\:pAq = 0$. This shows that $pq = 0 \Rightarrow \mathrm{dim}\:pAq = 0$. Next, assume that $pq \neq 0$ for some rank one projections $p$ and $q$ of $A$. Now, $\left(pqp-pq\right)^{2} = \left(qpq-pq\right)^{2} = 0$, and so, by hypothesis, $pqp = pq$ and $qpq = pq$. Hence, by the minimality of $p$ and $q$ and the fact that $pq \neq 0$, we get $p = \alpha q$ for some $\alpha \in \mathbb{C}-\left\{0 \right\}$. This shows that $pq \neq 0 \Rightarrow p = \alpha q$ for some $\alpha \in \mathbb{C}$. By Lemma \ref{b5.3} we conclude that $\mathrm{Soc}\:A \subseteq Z(A)$, establishing the result.   
\end{proof}

\begin{lemma}\label{b5.5}
$\mathrm{dim}\:aAa = \mathrm{rank}\,(a)$ for all $a \in \mathrm{Soc}\:A$ if and only if for all rank one projections $p$ and $q$ of $A$, $pq= 0 \Rightarrow \mathrm{dim}\:pAq = 0$.
\end{lemma}

\begin{proof}
Suppose first that $\mathrm{dim}\:aAa = \mathrm{rank}\,(a)$ for all $a \in \mathrm{Soc}\:A$. Let $p$ and $q$ be any rank one projections such that $pq = 0$. We claim that $\mathrm{dim}\:pAq = 0$: Suppose this is false. Then there exists an $x_{0} \in pAq$ such that $px_{0}q \neq 0$. Since $pq = 0$, it follows that $p \neq \alpha q$ for all $\alpha \in \mathbb{C}$. Moreover, by the subadditivity of the rank we have $\mathrm{rank}\,(p+q) \leq 2$. So, by hypothesis, $\left\{p, q \right\}$ is a spanning set of $(p+q)A(p+q)$ since $(p+q)p(p+q)$, $(p+q)q(p+q)$ and $(p+q)^{2}$ in $(p+q)A(p+q)$ implies that $p, q \in (p+q)A(p+q)$. Consequently, we can find $\lambda, \beta \in \mathbb{C}$ such that
$$\lambda p + \beta q = (p+q)x_{0}(p+q) = px_{0}p +px_{0}q + qx_{0}p + qx_{0}q.$$
However, then
$$px_{0}q = p^{2}x_{0}q^{2} = \lambda pq + \beta pq -px_{0}pq - pqx_{0}pq - pqx_{0}q = 0,$$
contradicting our choice of $x_{0}$. This proves our claim. So $pq = 0$ implies $\mathrm{dim}\:pAq = 0$. For the converse, let $a \in \mathrm{Soc}\:A - \left\{0 \right\}$ be arbitrary. By the density of $E(a)$ and the Diagonalization Theorem there exist a $u \in G(A)$, mutually orthogonal rank one projections $p_{1}, \ldots, p_{n}$ and $\lambda_{1}, \ldots, \lambda_{n} \in \mathbb{C}-\left\{0 \right\}$ such that $ua = \lambda_{1}p_{1} + \cdots +\lambda_{n}p_{n}$, where $\mathrm{rank}\,(a) = n \geq 1$. Observe that $uaAua = pAp$, where $p = p_{1} + \cdots +p_{n}$, by the orthogonality of the $p_{i}$. Moreover, by hypothesis, $\mathrm{dim}\:p_{i}Ap_{j} = 0$ for $i \neq j$. Hence,
$$\mathrm{dim}\:pAp \leq \mathrm{dim}\:\left(\sum_{i=1}^{n}\sum_{j=1}^{n}p_{i}Ap_{j}\right) \leq \sum_{i=1}^{n}\sum_{j=1}^{n}\mathrm{dim}\:p_{i}Ap_{j} = n,$$
where the final equality follows from the minimality of the $p_{i}$. Furthermore, by orthogonality it follows that $\left\{p_{1}, \ldots, p_{n} \right\}$ is a linearly independent set contained in $pAp$. Consequently,
$$ \mathrm{dim}\:aAa = \mathrm{dim}\:uaAua = \mathrm{dim}\:pAp = n = \mathrm{rank}\,(a),$$
as desired. This completes the proof.
\end{proof}

\begin{theorem}\label{b5.6}
$\mathrm{Soc}\:A \subseteq Z(A)$ if and only if $\mathrm{dim}\:aAa = \mathrm{rank}\,(a)$ for all $a \in \mathrm{Soc}\:A$.
\end{theorem}

\begin{proof}
The forward implication readily follows from Lemma \ref{b5.5}. Conversely, suppose that $\mathrm{dim}\:aAa = \mathrm{rank}\,(a)$ for all $a \in \mathrm{Soc}\:A$. Let $x \in \mathrm{Soc}\:A$ with $x^{2} = 0$ be arbitrary. We claim that $x = 0$: Suppose this is false. By the density of $E(x)$ and the Diagonalization Theorem there exist $u, v \in G(A)$, mutually orthogonal rank one projections $p_{1}, \ldots, p_{n}$, mutually orthogonal rank one projections $q_{1}, \ldots, q_{n}$, and $\alpha_{1}, \ldots, \alpha_{n}, \lambda_{1}, \ldots, \lambda_{n} \in \mathbb{C}-\left\{0 \right\}$ such that
$ux = \alpha_{1}p_{1} + \cdots +\alpha_{n}p_{n}$ and $xv = \lambda_{1}q_{1} + \cdots +\lambda_{n}q_{n}$, where $\mathrm{rank}\,(x) = n \geq 1$. It is useful to recall here that the $p_{i}$ all belong to $Aux \cap uxA$, and similarly, that the $q_{i}$ all belong to $Axv \cap xvA$. Hence, for each $i \in \left\{1, \ldots, n \right\}$ there exist $x_{i}, y_{i} \in A$ such that $p_{i} = x_{i}ux$ and $q_{i} = xvy_{i}$. Hence, $p_{j}q_{i} = 0$ for all $i, j \in \left\{1, \ldots, n \right\}$. Thus, for any $i, j \in \left\{1, \ldots, n \right\}$, we have $p_{j} \neq \alpha q_{i}$ for all $\alpha \in \mathbb{C}$. We now show that $q_{i}p_{j} = 0$ for all $i, j \in \left\{1, \ldots, n \right\}$: Let $i, j \in \left\{1, \ldots, n \right\}$ be arbitrary but fixed. Since $\left(p_{j} + q_{i}\right)^{2}$, $\left(p_{j} + q_{i}\right)p_{j}\left(p_{j} + q_{i}\right)$ and $\left(p_{j} + q_{i}\right)q_{i}\left(p_{j} + q_{i}\right)$ are all in $\left(p_{j} + q_{i}\right)A\left(p_{j} + q_{i}\right)$, it follows that $p_{j}, q_{i}, q_{i}p_{j} \in \left(p_{j} + q_{i}\right)A\left(p_{j} + q_{i}\right)$. Moreover, by the subadditivity of the rank we get $\mathrm{rank}\,\left(p_{j} +q_{i}\right) \leq 2$. Hence, by hypothesis, we have $q_{i}p_{j} = \beta p_{j} + \lambda q_{i}$ for some $\beta, \lambda \in \mathbb{C}$. Consequently, since $q_{i}p_{j}q_{i} = \beta p_{j}q_{i} + \lambda q_{i}$, $p_{j}q_{i}p_{j} = \beta p_{j} + \lambda p_{j}q_{i}$ and $p_{j}q_{i} = 0$, we get $\lambda q_{i} = 0$ and $\beta p_{j} = 0$. Hence, $q_{i}p_{j} = 0$. So, by Lemma \ref{b5.5} we may infer that $\mathrm{dim}\:q_{i}Ap_{j} = 0$ for all $i, j \in \left\{1, \ldots, n \right\}$. However, this means that for any $y \in A$ we have
$$xyx = xvv^{-1}yu^{-1}ux \in \sum_{i=1}^{n}\sum_{j=1}^{n}q_{i}Ap_{j} = \left\{0 \right\}.$$
But $x$ is von Neumann regular, so $x = 0$. This contradiction now proves our claim. Hence, $\mathrm{Soc}\:A \subseteq Z(A)$ by Theorem \ref{b5.4}. This establishes the result.
\end{proof}

\begin{theorem}\label{b5.7}
$\mathcal{C} = \left\{0 \right\}$ if and only if $\mathrm{Soc}\:A \subseteq Z(A)$.
\end{theorem}

\begin{proof}
The reverse implication is obvious. For the forward implication, let $p$ be a minimal projection of $A$. By hypothesis, for all $x, y \in A$ it follows that $\left[xp, py\right] = 0$. Hence, $xpy \in pAp$ for all $x, y \in A$. Thus, $J_{p}$ is one-dimensional and equals $pAp$. Consequently, every element of $\mathrm{Soc}\:A$ is of the form $\lambda_{1}p_{1} + \cdots +\lambda_{n}p_{n}$ for some minimal projections $p_{1}, \ldots, p_{n}$ and complex numbers $\lambda_{1}, \ldots, \lambda_{n}$. Next pick $x \in A$ and let $p$ be any projection in $\mathrm{Soc}\:A$. Since $\left[p, px\right] = \left[xp, p\right] = 0$, we infer that $xp = pxp = px$. Thus, $\mathrm{Soc}\:A \subseteq Z(A)$. 
\end{proof}

To conclude we show that the inequality
$$\mathrm{rank}\,(p) \leq \mathrm{dim}\:pAp \leq \left[\mathrm{rank}\,(p)\right]^{2}$$
for all $p = p^{2} \in \mathrm{Soc}\:A$ can be ``solved" at the extremities:

\begin{theorem}\label{b5.8}
$\left.\right.$
\begin{itemize}
\item[(1)]
$\mathrm{dim}\:pAp = \mathrm{rank}\,(p)$ for all finite-rank projections $p$ of $A$ if and only if $\mathrm{Soc}\:A \subseteq Z(A)$.
\item[(2)]
$\mathrm{dim}\:pAp = \left[\mathrm{rank}\,(p)\right]^{2}$ for all finite-rank projections $p$ of $A$ if and only if the Generalized Shoda's Theorem holds for $A$.
\end{itemize}
\end{theorem}

\begin{proof}
We firstly consider the characterization which appears as (1) above. By Theorem \ref{b5.6} it suffices to show that for each $a \in \mathrm{Soc}\:A$, there exists a projection $p \in \mathrm{Soc}\:A$ such that $\mathrm{dim}\:aAa = \mathrm{dim}\:pAp$ and $\mathrm{rank}\,(a) = \mathrm{rank}\,(p)$. But this is exactly what was shown in the first part of the proof of Theorem \ref{b4.5}. Next we consider the characterization which appears as (2) above. By Theorem \ref{b4.3} and Theorem \ref{b4.5}, it will suffice to show that $\mathrm{dim}\:pAp = \left[\mathrm{rank}\,(p)\right]^{2}$ for all finite-rank projections $p$ of $A$ if and only if $pAp \cong M_{n_{p}} \left(\mathbb{C}\right)$ for all finite-rank projections $p$ of $A$: The reverse implication is obvious since $p$ is the identity of $pAp$, and the forward implication follows from the last part of the argument in the proof of Theorem \ref{b4.5}. This completes the proof.
\end{proof}

Theorem \ref{b5.8} suggests that the dimension of certain subalgebras of the socle is to some extent a measure of commutativity.\\

\bibliographystyle{amsplain}
\bibliography{Spectral}

\end{document}